\documentclass[a4paper,9pt]{article}
\usepackage{amsmath, amsthm, amsfonts}
\usepackage{amssymb}
\usepackage{latexsym}
\usepackage{verbatim}

\usepackage[all,cmtip]{xy}
\usepackage{hyperref}

\input{xypic}

\usepackage{enumerate}
\usepackage{amssymb}
\usepackage{array}

\DeclareMathAlphabet{\mathfr}{U}{euf}{m}{n}

\newtheorem{teo}{Theorem}[section]

\newtheorem{rem}{Remark}[section]

\newtheorem{lema}{Lemma}[section]
\newtheorem{coro}{Corollary}[section]
\newtheorem{prop}{Proposition}[section]

\newcommand{\Q}{\mathbb Q}

\newcommand{\Gal}{\mathrm{Gal}}
\newcommand{\R}{\mathbb R}
\newcommand{\Z}{\mathbb Z}

\newcommand{\F}{\mathbb F}
\newcommand{\GL}{\mathrm{GL}}
\newcommand{\PGL}{\mathrm{PGL}}
\newcommand{\End}{\operatorname{End}}
\newcommand{\Hom}{\operatorname{Hom}}
\newcommand{\Frob}{\operatorname{Frob}}
\newcommand{\Aut}{\operatorname{Aut}}
\newcommand{\Ind}{\operatorname{Ind}}
\newcommand{\Inf}{\operatorname{Inf}}
\newcommand{\Log}{\operatorname{Log}}

\newcommand{\E}{\operatorname{E}}

\newcommand{\Res}{\operatorname{Res}}
\newcommand{\Tr}{\operatorname{Tr}}
\newcommand{\Ker}{\operatorname{Ker}}

\numberwithin{equation}{section}
\title{Artin representations attached to pairs of isogenous abelian varieties}
\author{Francesc Fit\'e}
\date{\today}
\begin{document}
\maketitle

\begin{abstract} Given a pair of abelian varieties defined over a number field $k$ and isogenous over a finite Galois extension $L/k$, we define a rational Artin representation of the group $\Gal(L/k)$ that shows a global relation between the $L$-functions of each variety and provides certain information about their decomposition up to isogeny over $L$. We study several properties of these Artin representations. As an application, for each curve $C'$ in a family of twists of a certain genus $3$ curve $C$, we explicitly compute the Artin representation attached to the Jacobians of $C$ and $C'$ and show how this Artin representation can be used to determine the $L$-function of the curve $C'$ in terms of the $L$-function of $C$. Moreover, from this Artin representation, we are able to compute the moments of the Sato-Tate distributions of the traces of the local factors of the curve $C'$.
\end{abstract}

\section{Introduction}

Let $A$ be an abelian variety of dimension $g$ defined over a number field $k$. Fix an algebraic closure $\overline k$ of $k$. All the extensions of $k$ that we will consider will be contained in $\overline k$. For a prime $\ell$, denote by $T_\ell(A)$ the $\ell$-adic Tate module of $A$ and write $V_\ell(A)=T_\ell(A)\otimes \Q$.  There is an action of the absolute Galois group $G_k=\Gal(\overline k/k)$ on $V_\ell(A)$, which gives an $\ell$-adic representation of dimension $2g$
$$\varrho_A \colon G_k\rightarrow \Aut_{\Q_\ell}(V_\ell(A))\simeq\GL_{2g}(\Q_\ell)\,.$$
As $\ell$ varies, the family $\{\varrho_A \}$ constitutes a strictly compatible system of rational $\ell$-adic representations in the sense of \cite{Ser89}. Let $\overline{\mathfrak{p}}$ be a prime of $\overline k$ lying over a prime~$\mathfrak{p}$ of $k$ not dividing $\ell$, let $I_{\overline{\mathfrak{p}}}$ be its inertia group and let $\Frob_{\overline{\mathfrak{p}}}$ be a Frobenius element over $\mathfrak{p}$. The polynomial
$$L_\mathfrak{p}(A/k,T)=\det(1-\varrho_A (\Frob_{\overline{\mathfrak{p}}})T;V_\ell(A)^{I_{\overline{\mathfrak{p}}}})$$
has integer coefficients and does not depend on $\ell$. A prime $\mathfrak{p}$ of $k$ is a prime of good reduction for $A$ if and only if the action of $I_{\overline{\mathfrak{p}}}$ through $\varrho_A $ is trivial on $V_\ell(A)$ (see \cite{ST68}). In this case, one has $L_\mathfrak{p}(A/k,T)=\prod_{i=1}^{g}(1-\alpha_iT)(1-\overline\alpha_iT)$, where the $\alpha_i$ are complex numbers such that $\alpha_i\overline\alpha_i=N \mathfrak{p}$. Here $N\mathfrak p$ stands for the norm of $\mathfrak p$, i.e., the size of the residue field of $\mathfrak p$. The $L$-function of $A/k$ is defined as
$$L(A/k,s)=\prod_\mathfrak{p} L_{\mathfrak{p}}(A/k,N\mathfrak p ^{-s})^{-1}\,,$$
where the product runs over all primes in $k$.

Let $A'$ be an abelian variety defined over $k$ and assume that there exists an isogeny $\phi\colon A\rightarrow A'$ defined over a finite Galois extension $L/k$.  By functoriality,~$\phi$ induces an isomorphism $V_\ell(A)\rightarrow V_\ell(A')$ of $\Q_\ell[G_L]$-modules. Let $\mathfrak{P}$ be a unramified prime of $L$ of norm $N\mathfrak{P}=(N\mathfrak{p})^f$. Consider
$$L_\mathfrak p(A/k,T)=\prod_{i}(1-\alpha_iT)\,,\qquad L_\mathfrak{p}(A'/k,T)=\prod_{i}(1-\beta_iT)\,.$$
Recall that the local factors $L_\mathfrak{P}(A/L,T)$ and $L_\mathfrak{P}(A'/L,T)$ coincide. Since $\mathfrak P$ is a unramified prime of $L$, $V_\ell(A)^{I_{\overline{\mathfrak p}}}=V_\ell(A)^{I_{\overline{\mathfrak P}}}$ and $V_\ell(A')^{I_{\overline{\mathfrak p}}}=V_\ell(A')^{I_{\overline{\mathfrak P}}}$. Moreover, since a Frobenius element $\Frob_{\overline{\mathfrak P}}$ over $\mathfrak{P}$ equals the $f$-power of a Frobenius element $\Frob_{\overline{\mathfrak p}}$ over $\mathfrak{p}$, we have
$$\prod_{i}(1-\alpha_i^fT)=L_\mathfrak{P}(A/L,T)=L_\mathfrak{P}(A'/L,T)=\prod_{i}(1-\beta_i^fT)\,.$$
By reordering the roots, if necessary, we obtain that for $i=1,\dots,\dim V_\ell(A)^{I_{\overline{\mathfrak p}}}$ one has that $\alpha_i^f=\beta_i^f$, i.e., the roots $\alpha_i$ and $\beta_i$ differ by a root of unity. The aim of this article is to study to what extent these roots of unity can be seen as the eigenvalues of the images of a certain Artin representation of the group $\Gal(L/k)$. We will make this idea more precise below. For this aim, it will be helpful  to give an alternative argument of the existence of these roots of unity.

Let $G$ denote $\Gal(L/k)$ until the end of this section. The semisimple module $\Q[G]$ decomposes as a direct sum $\bigoplus_\varrho\Q[G]_\varrho$ indexed by the rational irreducible representations of $G$, where~$\Q[G]_\varrho\simeq n_\varrho\cdot\varrho$ denotes the $\varrho$-isotypic component of $\Q[G]$. It follows from \cite{MRS07}, theorem $4.5$, that for every prime $\mathfrak p$ of~$k$ one has that
\begin{equation}\label{remid}
L_\mathfrak{p}(\Res^L_kA/k,T)=\prod_\varrho
L_{\mathfrak{p}}(A/k,\varrho,T)^{n_\varrho}\,,
\end{equation}
where the product on the right-hand side of the equality runs through all rational irreducible representations of $G$. For $\mathfrak p$ a prime of $k$ which is unramified in~$L$, the Rankin-Selberg polynomial $L_{\mathfrak{p}}(A/k,\varrho,T)$ is the  polynomial whose roots are all the products of roots of $L_{\mathfrak{p}}(A /k,T)$ and roots of $\det(1-\varrho(\Frob_{\mathfrak p})T)$. Since
$L_\mathfrak{p}(\Res^L_kA/k,T)$ and $L_\mathfrak{p}(\Res^L_kA'/k,T)$ coincide, we deduce that $L_{\mathfrak{p}}(A'/k,T)$ divides $\prod_\varrho L_{\mathfrak{p}}(A/k,\varrho,T)^{n_\varrho}$ and we thus recover the existence of the already mentioned roots of unity.

Moreover, identity (\ref{remid}) can be understood in terms of the Tate modules, asserting that
\begin{equation}
V_\ell(\Res^L_kA)\simeq \bigoplus_\varrho n_\varrho\cdot\varrho\otimes V_\ell(A)\simeq \Q[G]\otimes V_\ell(A)
\end{equation}
are isomorphic as $\Q_\ell[G_k]$-modules. Since $V_\ell(\Res^L_kA)$ and $V_\ell(\Res^L_kA')$ are isomorphic, the previous conclusion can now be rephrased by saying that $V_\ell(A')$ is a sub-$\Q_\ell[G_k]$-module of
$ \Q[G]\otimes V_\ell(A)$.
It now arises the question of wether a rational representation $\theta$ of~$G$ of dimension smaller than~$|G|$ can be defined satisfying
\begin{equation}\label{propprin}
V_\ell(A')\subseteq \theta\otimes V_\ell(A)
\end{equation}
as $\Q_\ell[G_k]$-modules.

We now present a situation, where such a representation always exists. Consider again the decomposition $\Q[G]\simeq\bigoplus_\varrho\Q[G]_\varrho$. For a rational irreducible representation~$\varrho$ of~$G$, let $\mathcal I_\varrho$ be $\Q[G]_\varrho\cap\Z[G]$. In \cite{MRS07}, attached to $\mathcal I_\varrho$, an abelian variety
$$A_\varrho:=\mathcal I_\varrho\otimes_\Z A$$
defined over $k$ is constructed. This construction is given in the context of commutative algebraic groups, but we will only be concerned with abelian varieties (for several explicit examples of this construction we refer the reader to \cite {Sil08}). It is shown that $A_\varrho$ is isomorphic over $L$ to $A^{n_\varrho\cdot\dim\varrho}$. Now property (\ref{propprin}) holds if one takes $\theta=\varrho$. Indeed, one has
$$V_\ell(A_\varrho)\simeq n_\varrho\cdot \varrho\otimes V_\ell(A)\subseteq \varrho\otimes V_\ell(A^{n_\varrho\cdot\dim\varrho})\,,$$
where for the first isomorphism we have applied theorem $2.2$ in \cite{MRS07}.
In particular, for $E$ an elliptic curve and $E_\chi$ the twist given by the quadratic character $\chi$ of a quadratic extension $L/k$, one can take $\theta=\chi$.

In section $2$, we define a rational Artin representation $\theta (A,A';L/k)$ and, for every prime $\ell$, a $\Q_\ell[G_k]$-module $W_\ell$ such that the representation $\theta_\ell(A,A';L/k)$ that affords satisfies that $\theta_\ell(A,A';L/k)\simeq \Q_\ell\otimes \theta(A,A';L/k)$.

We start section $3$ proving that $\theta (A,A';L/k)$ satisfies property (\ref{propprin}) as a consequence of the fact that $V_\ell(A')$ is a submodule of $W_\ell\otimes V_\ell(A)$. Besides, we investigate several properties of $\theta (A,A';L/k)$, such as its behavior under the change of the field extension, a characterization of its faithfulness, etc. In particular, we show that when we consider the representation $\theta (A_\varrho,A^{n_\varrho\cdot\dim\varrho};L/k)$ corresponding to the particular case of the twisted abelian varieties $A^{n_\varrho\cdot\dim\varrho}$ and $A_\varrho$, one recovers a multiple of the original representation $\varrho$.

Finally, in section $4$, as an example of the tools developed, we use this Artin representation to compute the $L$-functions of a family of twists $C'$ of a certain modular genus $3$ curve $C$. The \emph{global} control of the roots of unity relating the local factors of $C$ and $C'$ permits to compute the moments of the Sato-Tate distributions of the traces of $C'$. This is not a special feature of the particular example considered, but rather applies in a much wider context (see for example \cite{FS12}). 

\textbf{Acknowledgements}. I want to sincerely thank J.-C. Lario for providing many inspiring ideas, and offering constant help and support. I am also grateful to X. Guitart and J. Quer for carefully reading a preliminary version of this article. I was finantially supported by grants 2009 SGR 1220 and MTM2009-13060-C02-01.

\section{Definition and rationality }

The representation $\varrho_A $ endows $V_\ell(A)$ with a structure of $\Q_\ell[G_k]$-module, as we have already mentioned. We will denote by $V_\ell(A)^*$ its dual, that is the $\Q_\ell[G_k]$-module with underlying $\Q_\ell$-vector space $\Hom_{\Q_\ell}(V_\ell(A),\Q_\ell)$ endowed with the action given by $\varrho_A ^*$, the contragredient representation of $\varrho_A $.

Let $G$ be a group, $F$ a field and~$V$ an~$F$-vector space of dimension $n$. For a representation $\varrho\colon G\rightarrow \Aut_F(V)$, its contragredient representation $\varrho^*$ satisfies $(\varrho^*(\sigma)\lambda)(v)=\lambda(\varrho(\sigma^{-1})v)$ for any $\sigma \in G$, $\lambda\in\Hom_{F}(V,F)$ and $v\in V$. By choosing a basis in $V$, we have $\Aut_F(V)\simeq \GL_{n}(F)$ and we can see $\varrho(\sigma)$ as a matrix with coefficients in $F$. By taking the dual basis in $\Hom_{F}(V,F)$ of the taken basis in $V$, one has that $\varrho^*(\sigma)=(\varrho(\sigma)^{-1})^t$, where ${}^t$ indicates transposition of matrices.

Consider now the $\Q_\ell[G_k]$-module $V_\ell(A)^*\otimes V_\ell(A')$, where the action of $G_k$ is given by $\varrho_A ^*\otimes\varrho_{A'} $ and denote by $$W_\ell=(V_\ell(A)^*\otimes V_\ell(A'))^{G_L}$$ the subspace of the elements invariant under the action of the subgroup $G_L=\Gal(\overline k/L)$. In general, for $F$ a field, $G$ a group, $H$ a normal subgroup of~$G$, and~$V$ an~$F[G]$-module (with the action of $G$ on $V$ denoted by left exponentiation) and $v\in V^H$, $h\in H$, and $g\in G$, we have ${}^{h}(^gv)={}^g(^{g^{-1}hg}v)={}^g(^{h'}v)={}^gv$, where $h'=g^{-1}hg$ is an element of $H$. That is, $V^H$ is endowed with a structure of $F[G/H]$-module. As a consequence, we have the following lemma.

\begin{lema}\label{xor} The space $W_\ell$ acquires a structure of $\Q_\ell[\Gal(L/k)]$-module.
\end{lema}

We will denote by $\theta_\ell(A,A';L/k)$ the representation of $\Gal(L/k)$ afforded by the module $W_\ell$. Observe that for every prime $\mathfrak p$ of good reduction for $A$ and $A'$, the eigenvalues of $\theta_\ell(A,A';L/k)(\Frob_{\mathfrak{p}})$ are roots of unity obtained as quotients of roots of $L_\mathfrak p(A'/k,T)$ and roots of $L_\mathfrak p(A/k,T)$.

To investigate the properties of this representation, we need to recall some basic properties of tensor products, duals and morphisms of modules. We summarize them in the auxiliary result below. For $G$ a group and $F$ a field, let $V$ and $W$ be $F[G]$-modules. The $F$-vector space $\Hom_F(V,W)$ becomes an $F[G]$-module via the action defined as follows: for $g\in G$, $\lambda\in \Hom_F(V,W)$ and $v\in V$, we have
\begin{equation}\label{act}
(^g\lambda)(v)={}^g\lambda(^{g^{-1}}v)\,.
\end{equation}
Note that by taking $W=F$ with the trivial action of $G$, the module $\Hom_F(V,W)$ is just the dual module of $V$.
\begin{lema}\label{ret} Let $V$, $W$, $W_1$ and $W_2$ be $F[G]$-modules, which are of finite dimension as $F$-vector spaces. Let $H$ be a normal subgroup of $G$. We have the following isomorphisms:
\begin{enumerate}[i)]
\item $V^*\otimes W \simeq \Hom_F(V,W)$ as $F[G]$-modules.
\item $\Hom_{F}(V,W_1\otimes W_2)\simeq \Hom_{F}(V\otimes W_1^*,W_2)$ as $F[G]$-modules.
\item $(V^H)^*\simeq (V^*)^H$ as $F[G/H]$-modules.
\end{enumerate}
\end{lema}

\begin{proof}
For $i)$ we refer to lemma 3.12 of \cite{Kar92}. Although throughout this reference $G$ is always taken to be finite, we emphasize that the proof of this fact makes no use of this condition. Assertion $ii)$ follows from $i)$ and the fact that $(V\otimes W)^*\simeq V^*\otimes W^*$ (see lemma 3.11 of \cite{Kar92}):
$$
\begin{array}{l@{\,\simeq\,}l}
\Hom_{F}(V,W_1\otimes W_2) & \displaystyle{V^*\otimes (W_1\otimes W_2)}\\[6pt]
 & \displaystyle{(V\otimes W_1^*)^*\otimes W_2}\\[6pt]
 & \displaystyle{\Hom_{F}(V\otimes W_1^*,W_2)\,.}\\[6pt]
\end{array}
$$
For $iii)$ we need a new description of the subspace of invariants. Let $\alpha\colon G\rightarrow K$ be a morphism of groups. Denote $\Ind_G^KV$ the $F[G]$-module $F[K]\otimes_{F[G]}V$, where we can view $F[K]$ as an $F[G]$-module thanks to the morphism $\alpha$. If $\alpha$ is injective, the  $F[K]$-module $\Ind_G^KW$ coincides with usual notion of induced module from the subgroup $G$ to $K$. If $\alpha$ is surjective, the  $F[K]$-module $\Ind_G^KV$ is isomorphic to the subspace of $V$ consisting of the elements invariant under $\Ker(\alpha)$ (see \cite{Ser77}, exercise 7.1). Let $K$ be $G/H$ and let $\alpha$ denote the natural projection. Since the operations of inducing and taking duals commute, we have
$$ (V^H)^* \simeq \Ind_{G}^{G/H}(V)^*\simeq \Ind_{G}^{G/H}(V^*)\simeq (V^*)^H\,.$$
\end{proof}

\begin{rem}\label{after} With the notation of the previous lemma, let $\Hom_{F[H]}(V,W)$ denote the $F$-vector space of homomorphisms from $V$ to $W$ which commute with the actions of $H$ on $V$ and $W$. Equivalently, $\Hom_{F[H]}(V,W)$ can be characterized as the space $(\Hom_F(V,W))^H$ of the invariant elements in $\Hom_F(V,W)$ under the action (\ref{act}) and so, as justified before lemma \ref{xor}, it acquires a structure of $F[G/H]$-module. It follows that $\Hom_{F[H]}(V,W)$ is isomorphic to $(V^*\otimes W)^H$ and that $\Hom_{F[H]}(V,W_1\otimes W_2)$ is isomorphic to $\Hom_{F[H]}(V\otimes W_1^*,W_2)$.
\end{rem}

\begin{lema}\label{selfdual}  The module $W_\ell$ is selfdual.
\end{lema}
\begin{proof} First, we claim that the module $V_\ell(A)^*\otimes V_\ell(A')$ is selfdual. Since $\Q_\ell$ has characteristic zero and $V_\ell(A)^*\otimes V_\ell(A')$ is semisimple, it suffices to check that for every $\sigma\in G_k$ one has $ \Tr\varrho_A ^*\otimes\varrho_{A'} (\sigma)=\Tr\varrho_A \otimes\varrho_{A'} ^*(\sigma)$ (see \cite{Ser89}, chapter~$1$, section $2.1$). Since $\Tr$ is a continuous function, by the Cebotarev Density Theorem, it is enough to check it for the Frobenius elements corresponding to primes of good reduction for $A$ and~$A'$. Let $\sigma=\Frob_{\mathfrak p}$ be such an element. Note that if $\Tr\varrho_A (\sigma)=\sum_i\alpha_i+\overline\alpha_i$, then
$$\Tr\varrho_A ^*(\sigma)=\sum_i\frac{1}{\alpha_i}+\frac{1}{\overline\alpha_i}=\sum_i\frac{\overline\alpha_i}{N\mathfrak p}+\frac{\alpha_i}{N\mathfrak p}=\frac{1}{N\mathfrak p}\Tr\varrho_A (\sigma)\,.$$
The claim then follows from the formula
$$\Tr\left(\varrho_A ^*\otimes\varrho_{A'}  (\sigma)\right)=\frac{1}{N\mathfrak p}\left(\Tr\varrho_A (\sigma)\cdot\Tr\varrho_{A'} (\sigma)\right)\,.$$
We deduce that $W_\ell=(V_\ell(A)^*\otimes V_\ell(A'))^{G_L}$ is selfdual by applying $iii)$ of lemma~\ref{ret}.
\end{proof}

As a consequence of the fact that $V_\ell(A)^*\otimes V_\ell(A')$ is selfdual, we deduce that $\theta_\ell(A,A';L/k)$ and $\theta_\ell(A',A;L/k)$ are isomorphic.

Let $m$ be the exponent of $\Gal(L/k)$. Observe that $\Tr\theta_\ell(A,A';L/k)$ belongs to $\Q(\zeta_m)$, where $\zeta_m$ stands for a primitive $m$-th root of unity. Then, a theorem of Brauer (see \cite{Ser77}, theorem $24$), conjectured by Schur, guarantees that the representation $\theta_\ell(A,A';L/k)$ can be realized over $\Q(\zeta_m)$, that is, there exists a representation $\tilde\theta_\ell(A,A';L/k)$ with matrices having coefficients in $\Q(\zeta_m)$ (thus, an Artin representation) such that
$$\theta_\ell(A,A';L/k)\otimes\overline\Q_\ell\simeq \tilde\theta_\ell(A,A';L/k)\otimes\overline\Q_\ell\,.$$
It follows from lemma \ref{selfdual}, that $\Tr\theta_\ell(A,A';L/k)$ belongs to $\R\cap\Q(\zeta_m)$. Next, we shall see that in fact $\Tr\theta_\ell(A,A';L/k)$ belongs to $\Q$. Even more, we shall prove that $\theta_\ell(A,A';L/k)$ can be realized over $\Q$.

Let $\Hom_L(A,A')$ stand for the $\Z$-module of homomorphisms from $A$ to~$A'$ which are defined over $L$ and denote the $\Q$-vector space $\Hom_L(A,A')\otimes \Q$ by $\Hom_L^0(A,A')$. Write $\End_L(A)$ for the ring $\Hom_L(A,A)$ and $\End_L^0(A)$ for $\End_L(A)\otimes\Q$. Observe that, since $A$ and $A'$ are defined over $k$, the group $\Gal(L/k)$ acts on $\Hom_L^0(A,A')$. Call $\theta(A,A';L/k)$ the rational representation afforded by the $\Q[\Gal(L/k)]$-module $\Hom_L^0(A,A')$.

\begin{prop}\label{rac} The representation $\theta_\ell(A,A';L/k)$ is realizable over $\Q$. More precisely, one has
$$\theta_\ell(A,A';L/k)\simeq \Q_\ell\otimes \theta(A,A';L/k)\,.$$
\end{prop}
\begin{proof}
The results of Faltings (see \cite{Fal83}) ensure that
\begin{equation}\label{falt}
\Hom_L^0(A,A')\otimes\Q_\ell\simeq \Hom_{\Q_\ell[G_L]}(V_\ell(A),V_\ell(A'))
\end{equation}
as $\Q_{\ell}[\Gal(L/k)]$-modules. Remark \ref{after} together with equation (\ref{falt}) gives that $\theta_\ell(A,A';L/k)$ is isomorphic over $\Q_\ell$ to $\theta(A,A';L/k)$.
\end{proof}

\begin{coro} The family $\{\theta_\ell(A,A';L/k)\}_\ell$ constitutes a strictly compatible system of rational $\ell$-adic representations of the group $\Gal(L/k)$, where the exceptional set of primes consists on the primes of $k$ ramified in $L$. In particular, $\Tr\theta_\ell(A,A';L/k)$ is a rational character of $\Gal(L/k)$ which does not depend on the prime $\ell$.
\end{coro}
\begin{proof} For $\mathfrak p$ a prime of $k$ unramified in $L$, proposition \ref{rac} implies
$$\det(1-\theta_\ell(A,A';L/k)(\Frob_{\mathfrak p})T)=\det(1-\theta(A,A';L/k)(\Frob_{\mathfrak p})T)\,,$$
which is a polynomial with rational coefficients. In particular, for every prime~$\ell$, one has
that $\Tr\theta_\ell(A,A';L/k)$ equals $\Tr\theta(A,A';L/k)\,,$ which is a rational character of $\Gal(L/k)$ which does not depend on the prime~$\ell$.
\end{proof}

\begin{rem}\label{dim} Since $A$ and $A'$ are isogenous over $L$, we have that
$$\dim\theta(A,A';L/k)=\dim_\Q\Hom_L^0(A,A')=\dim_\Q\End_L^0(A)>0\,.$$
\end{rem}
We will refer to $\theta(A,A';L/k)$ as the Artin representation attached to the isogenous abelian varieties $A$ and $A'$ over $L/k$.

\section{Properties of $\theta(A,A';L/k)$}

In the previous section, a rational Artin representation $\theta (A,A';L/k)$ has been attached to a pair of isogenous abelian varieties $A\sim_L A'$ over a finite Galois extension $L/k$. We show that this representation satisfies property (\ref{propprin}) in the Introduction.

\begin{teo}\label{bont} $V_\ell(A')$ is a sub-$\Q_\ell[G_k]$-module of
$\theta (A,A';L/k)\otimes V_\ell(A)\,.$
\end{teo}

\begin{proof} By proposition \ref{rac}, we have $\theta (A,A';L/k)\otimes V_\ell(A)\simeq W_\ell\otimes V_\ell(A)$. Suppose that $V_\ell(A')\simeq\bigoplus n_i\cdot V_i$ as a sum of simple $\Q_\ell[G_k]$-modules $V_i$. We want to show that $V_i$ appears with multiplicity at least $n_i$ in $W_\ell\otimes V_\ell(A)$. Since
$$W_\ell\otimes V_\ell(A)\simeq\bigoplus_jn_j\cdot(V_\ell(A)^*\otimes V_j)^{G_L}\otimes V_\ell(A)\,,$$
it suffices to see that $V_i$ is a constituent of $(V_\ell(A)^*\otimes V_i)^{G_L}\otimes V_\ell(A)$. Since $V_i$ is simple, this is equivalent to prove that
$\dim_{\Q_\ell}\Hom_{\Q_\ell[G_k]}(V_i,(V_\ell(A)^*\otimes V_i)^{G_L} \otimes V_\ell(A))>0\,.$ Now, $ii)$ of lemma \ref{ret} provides an isomorphism
$$\Hom_{\Q_\ell[G_k]}(V_i,(V_\ell(A)^*\otimes V_i)^{G_L} \otimes V_\ell(A))\simeq\Hom_{\Q_\ell[G_k]}(V_\ell(A)^*\otimes V_i,(V_\ell(A)^*\otimes V_i)^{G_L})\,,$$
from which the result follows provided that $(V_\ell(A)^*\otimes V_i)^{G_L}$ is a non-trivial sub-$\Q_\ell[G_k]$-module of $V_\ell(A)^*\otimes V_i$. Indeed, since $V_\ell(A)\simeq\bigoplus n_i\cdot V_i$ as $\Q_\ell[G_L]$-modules, one has
$$
\begin{array}{ll}
\dim_{\Q_\ell}(V_\ell(A)^*\otimes V_i)^{G_L} & \displaystyle{=\,\dim_{\Q_\ell}\Hom_{\Q_\ell[G_L]}(V_\ell(A), V_i)}\\[6pt]
 & \displaystyle{=\,\sum_j n_j\dim_{\Q_\ell}\Hom_{\Q_\ell[G_L]}(V_j, V_i)}\\[6pt]
 & \displaystyle{\geq\,\dim_{\Q_\ell} \End_{\Q_\ell[G_L]}(V_i)\,.}\\[6pt]
\end{array}
$$
\end{proof}

\begin{coro}\label{eldesc} Suppose that $A'$ is simple over $k$ and that $A\sim_k E^g$, for $E$ an elliptic curve defined over $k$ such that $\End^0_k(E)\simeq\Q$. Then, $V_\ell(A')\simeq \varrho\otimes V_\ell(E)$, where $\varrho$ is a representation of $\Gal(L/k)$ such that $\theta(A,A';L/k)=g\cdot\varrho$.
\end{coro}
\begin{proof}
Observe that the hypothesis $A\sim_k E^g$ implies that $\theta(A,A';L/k)=g\cdot\varrho$, for a certain rational representation $\varrho$. By theorem \ref{bont}, we have
$$V_\ell(A')\subseteq\theta(A,A';L/k)\otimes V_\ell(E^g)\simeq g^2\cdot \varrho\otimes V_\ell(E)\,.$$
Since $V_\ell(A')$ is simple, $V_\ell(A')\subseteq \varrho \otimes V_\ell(E)$. Since $\dim\varrho\leq g$ (as a consequence of $\End^0_k(E)\simeq\Q$), the inclusion is an isomorphism and the proposition follows.
\end{proof}

\begin{prop}\label{collde} Let $\varrho$ be an irreducible rational representation of  $\Gal(L/k)$. If $\End_L^0(A)\simeq\Q$, then $\theta(A^{n_\varrho\cdot\dim\varrho},A_\varrho;L/k)\simeq n_{\varrho}^2\cdot\dim\varrho\cdot\varrho$.
\end{prop}
\begin{proof} Observe that on the one hand, we have
$$(V_\ell(A^{n_\varrho\cdot\dim\varrho})^*\otimes V_\ell(A_\varrho))^{G_L}\simeq n_\varrho^2\cdot\dim\varrho\cdot\varrho\otimes(V_\ell(A)^*\otimes V_\ell(A))^{G_L}\,.$$
On the other hand, $\End_L^0(A)\simeq\Q$ implies that $(V_\ell(A)^*\otimes V_\ell(A))^{G_L}$ is the trivial representation.
\end{proof}

In particular, if $\chi$ is the quadratic character of a quadratic extension $L/k$, $E$ is an elliptic curve defined over $k$ such that $\End_L^0(E)\simeq\Q$ and $E_\chi$ the quadratic twist of $E$ given by $\chi$, then $\theta(E,E_\chi;L/k)\simeq\chi$.

So far we have not paid much attention to the extension $L/k$. We do this in the following, where we are concerned with the problem of relating the distinct representations attached to the pair $A$ and $A'$ obtained when we let the field of definition of the isogeny between them vary. First we remind some notations. Let $G$ be a group and $H$ be a subgroup of $G$. If $\varrho$ is a representation of $G$, we denote by $\Res_H^G\varrho$ the representation $\varrho$ restricted to the subgroup $H$. If $\varrho$ is a representation of  $H$, we denote by $\Ind_H^G\varrho$ the induced representation from $H$ to $G$. If $H$ is normal in $G$, $\varrho$ is a representation of $G/H$ and $\pi$ is the canonical projection from $G$ to $G/H$, we denote by $\Inf_{G/H}^G\varrho$ the representation $\varrho\circ\pi$. If $F/F''$ is a Galois extensions of fields, $F'/F''$ a subextension of $F/F''$ and $G=\Gal(F/F'')$ and $H=\Gal(F/F')$, we will simply write $\Res^{F''}_{F'}$, $\Ind^{F''}_{F'}$ and $\Inf^F_{F'}$ for $\Res_H^G$, $\Ind_H^G$ and $\Inf_{G/H}^G$, respectively.

\begin{prop}\label{sera} One has:
\begin{enumerate}[i)]
\item Let $L'/k$ be a Galois subextension of $L/k$ over which $A$ and $A'$ are isogenous. Then, the representation $\theta(A,A';L/k)$ is isomorphic to a subrepresentation of $\Ind_{L'}^{k}\theta(A,A';L/L')$.
\item Let $L'/k$ be a Galois extension containing $L/k$. Then, the representation $\Inf_{L}^{L'}\theta(A,A';L/k)$ is isomorphic to a subrepresentation of $\theta(A,A';L'/k)$.
\end{enumerate}
\end{prop}
\begin{proof}  For the first statement, observe that $\theta(A,A';L/L')=\Res_{L'}^{k}\theta(A,A';L/k)$. Now we obtain the result from the fact that any complex representation $\varrho$ of a finite group $G$ is a subrepresentation of $\Ind_H^G\Res_H^G\varrho$, for any subgroup $H$ of~$G$. Indeed, let $\chi_\varrho$ denote the character of $\varrho$. By Frobenius reciprocity, for any irreducible character $\chi$ of $G$, one has
$$(\chi,\chi_\varrho)_G\leq(\Res_H^G\chi,\Res_H^G\chi_\varrho)_H=(\chi,\Ind_H^G\Res_H^G\chi_\varrho)_G,\,$$
where $(\cdot,\cdot)_G$ and $(\cdot,\cdot)_H$ denote the scalar products on complex-valued functions on $G$ and $H$, respectively. The second statement is due to the fact that for every $F[G]$-module $V$ and normal subgroup $H$ of $G$, $\Inf_H^GV^H$ is a sub-$F[G]$-module of $V$.
\end{proof}

\begin{prop}\label{injec} The representation $\theta(A,A';L/k)$ is faithful if and only if for every proper Galois subextension $L'/k$ of $L/k$ one has $$\dim_\Q\Hom_{L'}^0(A,A')<\dim_\Q\Hom_L^0(A,A')\,.$$
\end{prop}

\begin{proof} Indeed, $\theta(A,A';L/k)$ is not faithful if and only if there exists a proper Galois subextension $L'/k$ of $L/k$ such that $$\Hom_{L'}^0(A,A')=\Hom_L^0(A,A')\,.$$
Since the space on the left-hand side of the equality is always included into the one on the right-hand side, asking for equality of the spaces amounts to asking for equality of their dimensions.
\end{proof}

As a corollary of this proposition, we obtain that if there does not exist a Galois subextension of $L/k$ over which $A$ and $A'$ are isogenous, then $\theta(A,A';L/k)$ is faithful. Therefore, a subextension $L'/k$ of $L/k$ can always be taken so that $\theta(A,A';L'/k)$ is faithful. We give an elementary result, which will be useful to determine the irreducible constituents of $\theta(A,A';L/k)$.

\begin{prop}\label{constituents} Suppose that $\theta(A,A';L/k)\simeq\bigoplus n_\varrho\cdot \varrho$, where the sum runs over the absolutely irreducible representations of $\Gal(L/k)$ and the $n_\varrho$ are nonnegative integers. Let $L'/k$ be a subextension of $L/k$. Then we have
$$\sum_{\Ker\varrho\supseteq\Gal(L/L')}n_\varrho\cdot\dim\varrho=\dim_\Q\Hom_{L'}^0(A,A')\,.$$
In particular, $\dim_\Q\Hom_k^0(A,A')$ is the multiplicity of the trivial representation of $\Gal(L/k)$ in $\theta(A,A';L/k)$.
\end{prop}
\begin{proof} It is enough to observe that the space $\Hom_{L'}^0(A,A')$ is isomorphic to the direct sum of those constituents $\varrho$ of $\theta(A,A';L/k)$ such that $\Ker\varrho$ contains~$\Gal(L/L')$.
\end{proof}

To close this section, we turn to discuss a kind of transitivity property satisfied by $\theta(A,A';L/k)$. It will turn out to be a  very useful tool, as shown in the next section. Let $A_1$, $A_2$ and $A_3$ be abelian varieties defined over a number field $k$, such that $A_1$ and $A_2$ are isogenous over the finite Galois extension $L/k$ and $A_1$ and $A_3$ are isogenous over the finite Galois extension $L'/k$. Notice that $A_2$ and $A_3$ are isogenous over the composition $L  L'$.
$$
\xymatrix{
&A_1  \ar[rd]^{L'}\ar[ld]_{L}&\\
A_2 \ar[rr]^{LL' } && A_3 }
$$
Next proposition shows how the Artin representations attached to the three pairs of isogenous varieties are related.
\begin{prop}\label{trans} The representation $\theta(A_2,A_3;L  L'/k)$ is isomorphic to a subrepresentation of
$$\theta(A_1,A_2;LL'/k)\otimes\theta(A_1,A_3; LL'/k)\,.$$
\end{prop}

\begin{proof} Since $\theta(A_1,A_2;LL'/k)$ is selfdual, by $i)$ of lemma \ref{ret}, it suffices to prove that there is an inclusion of $\Q[\Gal(LL'/k)]$-modules
$$\Hom_{LL'}(A_2,A_3)\subseteq\Hom_{\Q}(\Hom_{LL'}(A_1,A_2),\Hom_{LL'}(A_1,A_3))\,.$$
Denote $\Hom_{LL'}(A_i,A_j)$ by $F_{ij}$. For each $\varphi\in F_{23}$, define $\tilde\varphi\in\Hom_{LL'}(F_{12},F_{13})$ in the following way: for each $\lambda\in F_{12}$, let $\tilde\varphi (\lambda)=\varphi\circ \lambda\in F_{13}$. The map $\varphi\mapsto\tilde\varphi$ gives an inclusion that respects the action of $\Gal(LL'/k)$. Indeed, one has
$$(\widetilde{{}^\sigma\varphi})(\lambda)={}^\sigma\varphi\circ\lambda={}^\sigma(\varphi\circ{}^{\sigma^{-1}}\lambda)={}^\sigma(\tilde\varphi({}^{\sigma^{-1}}\lambda))=({}^\sigma\tilde\varphi)(\lambda)\,,$$
for any $\sigma\in\Gal(LL'/k)$, $\varphi\in F_{23}$ and $\lambda\in F_{12}$.
\end{proof}

\section{Example}

In \cite{BFGL07}, the modular genus $3$ curve $C_0=X^+(7,3)$ is defined. Moreover, for each polynomial $f(x)\in\Q[x]$ of degree $4$ such that its splitting field $L$ is an $\mathcal S_4$-extension containing $\Q(\sqrt{-3})$ a way to produce a twist $X^+(7,3)_\varrho$ is shown. Here, $\varrho$ stands for a surjective Galois representation of $G_\Q$ onto $\PGL_2(\F_3)$ determined up to conjugation by its splitting field $L$. In the same article, the curve $C_0$ is shown to be isomorphic over $\Q(\sqrt{-3})$ to the plane quartic $C_1$ given by the following equation
$$X^4+Y^4+Z^4+\frac{2}{7}\left(X^2Z^2+Y^2Z^2+X^2Y^2\right)=0\,.$$
The Jacobian $J(C_0)$ of the curve $C_0$ happens to be $\Q$-isogenous to $E_{21}^2\times E_{63}$, where $E_{21}$ and $E_{63}$ stand for the elliptic curves of conductor $21$ and label $A1$, and conductor $63$ and label $A2$ in Cremona's Tables (see \cite{Cre97}). It can be checked that the curves $E_{21}$ and $E_{63}$ do not have complex multiplication and are isomorphic over $\Q(\sqrt{-3})$, but not over $\Q$. The Jacobian of the curve $C_1$ is $\Q$-isogenous to $E_{21}^3$.

Let $C_2=X^+(7,3)_\varrho$ and $C_3=X^+(7,3)_{\varrho'}$ be the curves attached to any two Galois representations $\varrho$ and $\varrho'$ from $G_\Q$ onto $\PGL_2(\F_3)$, with splitting fields $L$ and $L'$ satisfying $L\cap L'=\Q(\sqrt{-3})$. Let $f$ and $f'$ be defining polynomials of $L$ and $L'$, respectively. Let $\phi_{ij}$ stand for a fixed isomorphism from $C_i$ to $C_j$, and use the same notation to refer to the induced isomorphism between the Jacobians. It follows from the definition of $X^+(7,3)_\varrho$ that $\phi_{21}$ and $\phi_{13}$ are respectively defined over $L_{21}=L$ and $L_{31}=L'$. The isomorphism $\phi_{32}$ is clearly defined over the composition $L_{32}=L  L'$, which is an extension of degree $288$ over $\Q$. As mentioned above, the isomorphism $\phi_{01}$ can be defined over $L_{01}=\Q(\sqrt{-3})$. Since both $L$ and $L'$ contain $\Q(\sqrt{-3})$, it follows that $\phi_{02}$ and $\phi_{03}$ are defined over $L_{02}=L$ and $L_{03}=L'$, respectively.

$$
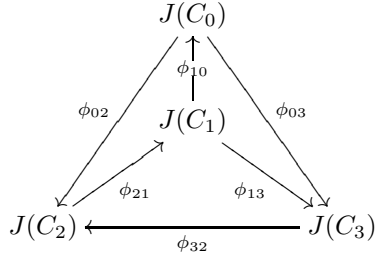
\xymatrix{
&J(C_0)\ar[rdd]^{\phi_{03}}\ar[ldd]_{\phi_{02}}&\\
&J(C_1)  \ar[rd]_{\phi_{13}}\ar[u]|{\phi_{10}}&\\
J(C_2) \ar[ru]_{\phi_{21}} && J(C_3) \ar[ll]^{\phi_{32}} }
$$

One of the aims of this section is to compute the Artin representations $\theta_{ij}=\theta(J(C_i),J(C_j);L_{ij}/\Q)$ for $(i,j)=(0,1)$, $(0,2)$, $(0,3)$, $(2,1)$, $(1,3)$, and $(3,2)$. Since $J(C_1)$ and $J(C_2)$ are $L$-isogenous to $E_{21}^3$, it follows from remark~\ref{dim} that $\theta_{21}$ has dimension $9$. The same argument proves that $\theta_{13}$ and $\theta_{32}$ have dimension $9$. We fix the following notation:

\begin{enumerate} [i)]
\item Let $1a_{L}$, $2a_{L}$,... stand for the conjugacy classes of $\Gal(L/\Q)$, and $\chi_1$, $\chi_2$,... for its irreducible characters (see Table \ref{ctL}). Note that the traces in the entries of Table \ref{ctL} are given as sums of eigenvalues of the corresponding irreducible representations. Analogously, denote by $1a_{L'}$, $2a_{L'}$,... the conjugacy classes of $\Gal(L'/\Q)$, and by $\chi_1'$, $\chi_2'$,... its irreducible characters. Moreover, let $\varrho_i$ be the irreducible representation associated to $\chi_i$.
\item Let $\chi_t$ and $\chi_q$ be the trivial and the quadratic characters of $\Gal(\Q(\sqrt{-3})/\Q)$, respectively.
\item Let $1A$, $2A$,... stand for the conjugacy classes of $\Gal(L  L'/\Q)$ and $\psi_1$, $\psi_2$,... for its irreducible characters (see Table \ref{ctLL'}).
\end{enumerate}

\begin{table}[h]
$$
\begin{array}{r|rrrrrrrrrrrrrr}
\rm{Class} &   1a_L & 2a_L & 2b_L & 3a_L & 4a_L \\\hline
\rm{Size}  &   1 & 3 & 6 & 8 & 6 \\\hline
\chi_1  &   1 & 1&  1&  1&  1 \\
\chi_2  &   1 & 1&  -1&  1& -1 \\
\chi_3  &   2 & 1+1&  1-1&  \zeta_3+\overline\zeta_3 &  1-1 \\
\chi_4  &   3 & 1-1-1&  1-1-1&  1+\zeta_3+\overline\zeta_3&  1+i-i \\
\chi_5  &   3 & 1-1-1&  1+1-1&  1+\zeta_3+\overline\zeta_3&  -1+i-i \\
\end{array}
$$
\caption{Character table of $\Gal(L/\Q)$} \label{ctL}
\end{table}

\begin{table}[h]
$$
\begin{array}{r|rrrrrrrrrrrrrr}
\rm{Class} &   1A & 2A & 2B & 2C & 2D & 3A & 3B & 3C& 3D & 4A & 4B & 4C& 6A & 6B \\\hline
\rm{Size}  &   1 & 2 & 2 & 2 & 2 & 3 & 3 & 3 & 3 & 4& 4 & 4 & 6 & 6 \\\hline
\psi_1  &   1 & 1&  1&  1&  1&  1&  1&  1&  1&  1&  1&  1&  1&  1\\
\psi_2  &   1 & 1&  1&  1& -1&  1&  1 & 1 & 1 &-1 &-1 &-1 & 1 & 1\\
\psi_3  &   2 & 2&  2&  2&  0& -1&  2 &-1 &-1 & 0 & 0 & 0 &-1 & 2\\
\psi_4  &   2 & 2&  2&  2&  0&  2& -1 &-1& -1&  0&  0&  0 & 2& -1\\
\psi_5  &   2 & 2&  2&  2&  0& -1& -1 & 2 &-1 & 0 & 0 & 0 &-1 &-1\\
\psi_6  &   2 & 2&  2&  2&  0& -1& -1 &-1&  2&  0 & 0 & 0 &-1& -1\\
\psi_7  &   3 & 3& -1& -1&  1&  0&  3 & 0 & 0 &-1 & 1 &-1 & 0& -1\\
\psi_8  &   3 & 3& -1& -1& -1&  0&  3 & 0 & 0 & 1 &-1 & 1 & 0& -1\\
\psi_9  &   3 &-1&  3& -1&  1&  3&  0 & 0  &0 &-1& -1 & 1& -1 & 0\\
\psi_{10} &  3& -1&  3& -1& -1&  3&  0&  0 & 0 & 1 & 1 &-1 &-1 & 0\\
\psi_{11} &  6& -2&  6& -2&  0& -3&  0&  0 & 0 & 0 & 0 & 0 & 1 & 0\\
\psi_{12} &  6&  6& -2& -2&  0&  0& -3&  0 & 0 & 0 & 0 & 0 & 0 & 1\\
\psi_{13} &  9& -3& -3&  1&  1&  0&  0&  0 & 0 & 1 &-1 &-1 & 0 & 0\\
\psi_{14} &  9& -3& -3&  1& -1&  0&  0&  0 & 0 &-1 & 1 & 1 & 0 & 0\\
\end{array}
$$
\caption{Character table of $\Gal(L  L'/\Q)$} \label{ctLL'}
\end{table}

Consider the subfields $L_3$ and $L_4$ of $L$ defined in \cite{BFGL07}. We recall that $L_4/\Q$ denotes a quartic extension generated by a root of the polynomial $f$ and~$L_3$ denotes a cubic extension generated by a root of the resolvent of $f$. In loc. cit. page~375, it is shown that there exist elliptic curves $E_R$ defined over $L_3$ and $E_S$ defined over $\Q$ such that
\begin{equation}\label{restric}
J(C_2)\sim_\Q \Res^{L_3}_\Q E_R\qquad\text{and}\qquad J(C_2)\times E_S\sim_\Q \Res^{L_4}_\Q E_S\,.
\end{equation}
According to loc. cit. page~371, the elliptic curve~$E_S$ is the one with conductor~21 and label~A3 in Cremona's Tables. Since there is a single isogeny class of conductor~21, we have that $E_S\sim_\Q E_{21}$.

\begin{lema}\label{calcfac} Let $p$ be a prime of good reduction for both $E_{63}$ and $J(C_2)$. Let $(1-\alpha T)(1-\overline \alpha T)$ be the local factor $L_p(E_{63}/\Q,T)$. Then, the local factor $L_p(J(C_2)/\Q,T)$ equals
$$
\begin{cases}
  L_p(E_{63}/\Q,T)(1-\zeta_3 \alpha T)(1-\overline\zeta_3\overline \alpha T)(1-\overline\zeta_3 \alpha T)(1-\zeta_3 \overline \alpha T) & \text{ if $f_{L_3}=3$}\\
  L_p(E_{63}/\Q,T)(1-i \alpha T)(1+i\overline \alpha T)(1+i \alpha T)(1-i\overline \alpha T)& \text{ if $f_{L_4}=4$,}
\end{cases}
$$
where $f_{L_3}$ and $f_{L_4}$ denote the residue class degrees of $p$ in $L_3$ and $L_4$, respectively, and $\zeta_3$ a primitive cubic root of unity.
\end{lema}

\begin{proof} Recall that for an abelian variety $A$ defined over a number field extension $L/k$ and a prime $\mathfrak p$ of $k$, we have the equality
$$L_{\mathfrak p}(\Res^L_kA/k,T)=\prod_{\mathfrak P|\mathfrak p} L_{\mathfrak P}(A/L,T^{f_{\mathfrak P}})\,,$$
where $f_\mathfrak P$ denotes the residue class degree of the prime $\mathfrak P$ of $L$. Thus, if $f_{L_3}=3$ and~$\mathfrak P_3$ denotes the prime of $L_3$ over $p$, relation (\ref{restric}) gives
\begin{equation}\label{punts}
L_p(J(C_2)/\Q,T)=L_{\mathfrak P_3}(E_R/L_3,T^3)=1-(1+p^3-\#\tilde E_R(\F_{p^3}))T+p^3T^6\,,
\end{equation}
where $\tilde E_R$ denotes the reduction of $E_R$ modulo the prime $\mathfrak P_3$. Since $E_R\simeq_L E_{63}$ and $f_{L_3}=3$ implies $f_{L}=3$, one has that $\#\tilde E_R(\F_{p^3})=\#\tilde E_{63}(\F_{p^3})$. Write $L_p(E_{63}/\Q,T)=1-aT+pT^2$.  By differentiating the function
$$\Log\left(\frac{1-aT+pT^2}{(1-T)(1-pT)}\right)=\sum_{n\geq 1}\#\tilde E_{63}(\F_{p^n})\frac{T^n}{n}\,,$$
one easily obtains that $\#\tilde E_{63}(\F_{p^3})=1+p^3-a^3+3ap$. Substituting in equation~(\ref{punts}) and factoring, yields
$$L_p(J(C_2)/\Q,T)=(1-aT+pT^2)(1+aT+(a^2-p)T^2+apT^3+p^2T^4)\,.$$
Now, a straightforward calculation of the roots of the second factor of the right hand side of the above expression leads us to the first statement of the lemma. For the second, if $\mathfrak P_4$ denotes the prime of $L_4$ over $p$, relation (\ref{restric}) and the fact that $E_S\sim_\Q E_{21}$ imply
$$L_p(J(C_2)/\Q,T)L_p(E_{21}/\Q,T)=L_{\mathfrak P_4}(E_{21}/L_4,T^4)\,.$$
Write again $L_p(E_{63}/\Q,T)=1-aT+pT^2$, and note that $f_{L_4}=4$ implies $f_L=4$ and $f_{\Q(\sqrt{-3})}=2$. Thus $L_p(E_{21}/\Q,T)=L_p(E_{63}/\Q,-T)=1+aT+pT^2$. Applying an analogous argument to the one in the previous case, one arrives at
$$L_{\mathfrak P_4}(E_{21}/L_4,T^4)=(1+aT+pT^2)(1-aT+pT^2)(1+(a^2-2p)T^2+p^2T^4)\,.$$
Now the lemma follows by computing the roots of the last factor of the above expression.
\end{proof}

\begin{lema}\label{coseta} The curves $E_{21}$ and $E_{63}$ do not appear in the decomposition up to isogeny over $\Q$ of $J(C_2)$.
\end{lema}
\begin{proof} Suppose that $J(C_2)\sim_\Q E_{21}\times A$ or $J(C_2)\sim_\Q E_{63}\times A$, for some abelian surface $A$ defined over $\Q$, and we will reach a contradiction. Then $A\sim_L E_{63}^2$. As explained in the previous sections, we can consider the Artin representation $\theta(A,E_{63}^2;L/\Q)$ and it has dimension $4$. Let $p$ be a non-supersingular prime for $E_{63}$ of good reduction for both $E_{63}$ and $A$ with residue class degree $f_{L_3}=3$. If we write $L_p(E_{63}/\Q,T)=(1-\alpha T)(1-\overline \alpha T)$, the condition of $p$ being non-supersingular guarantees that $\overline\alpha/\alpha$ has infinite order (see for example \cite{Tat66}, theorem $2$). From lemma \ref{calcfac}, it follows that the only possible values for a quotient of a root of $L_{p}(E_{63}/\Q,T)$ and a root of $L_{p}(A/\Q,T)$ that have finite order are $\zeta_3$ and $\overline\zeta_3$. Thus, $\theta(A,E_{63}^2;L/\Q)\simeq 2\cdot\varrho_3$ (see Table \ref{ctL}). Let $p$ be a non-supersingular prime for $E_{63}$ of good reduction for both $E_{63}$ and $A$ with residue class degree $f_{L_4}=4$. Lemma \ref{calcfac} shows that the only possible values for a quotient of a root of $L_{p}(E_{63}/\Q,T)$ and a root of $L_{p}(A/\Q,T)$ that have finite order are $i$ and $-i$. We reach a contradiction by observing that the  eigenvalues of $\varrho_3$ on the class $4a_L$ are are $1$ and $-1$ (see Table \ref{ctL}).
\end{proof}

We have seen that $E_{63}$ is not a $\Q$-factor of $J(C_2)$. Nevertheless, as we will prove later, for every prime $p$ of good reduction for both $J(C_2)$ and $E_{63}$, the polynomial $L_p(E_{63}/\Q,T)$ divides $L_p(J(C_2)/\Q,T)$. Another example where this curious phenomenon occurs is the following one. Let $E$ denote an elliptic curve over $\Q$ and let $d_1$ and $d_2$ be non-square rational numbers such that $d_1d_2$ is also a non-square. Consider the product $A=E_{d_1}\times E_{d_2}\times E_{d_1d_2}$ of the quadratic twists of $E$ by $d_1$, $d_2$ and $d_1d_2$. Then, for every prime $p$ of good reduction for both $E$ and $A$, it is clear that $L_p(E/\Q,T)$ divides $L_p(A/\Q,T)$, even though $E$ is not a $\Q$-factor of $A$.

\begin{coro} The Jacobian $J(C_2)$ is simple over $\Q$.
\end{coro}
\begin{proof} Assume $J(C_2)$ decomposes, i.e., $J(C_2)\sim_\Q E\times A$, for some elliptic curve $E$ and some abelian surface $A$, both defined over $\Q$. It follows that $E_{63}\sim_L E$. Let $\theta$ stand for $\theta(E,E_{63};L/\Q)$. Since $\dim\theta=1$ ($E$ does not have complex multiplication), we have that either $\Tr\theta=\chi_1$ or  $\chi_2$. In any case, $\theta$ factors through $\Q(\sqrt{-3})$, the only quadratic extension of $L$. Then, theorem \ref{bont} implies that $V_\ell(E)\simeq \theta \otimes V_\ell(E_{63})$, from which we deduce that either $E\sim_\Q E_{21}$ or $E\sim_\Q E_{63}$. But this is a contradiction with the previous lemma.
\end{proof}

We now compute all the Artin representations $\theta_{ij}$ involved in the above graph of isogenies. As for $\theta_{32}$, it will be computed in proposition \ref{megtrans} from $\theta_{21}$ and $\theta_{13}$ and the transitivity property stated in proposition \ref{trans}.

\begin{prop} We have:
\begin{enumerate}[i)]
\item $\Tr(\theta_{21})=3\cdot \chi_5 $
\item $\Tr(\theta_{13})=3\cdot \chi'_5 $
\item $\Tr(\theta_{10})=6\cdot \chi_t+ 3\cdot\chi_q $
\item $\Tr(\theta_{02})= \chi_4+2\cdot\chi_5 $
\item $\Tr(\theta_{03})= \chi'_4+2\cdot\chi'_5 $.
\end{enumerate}
\end{prop}

\begin{proof}
By proposition \ref{constituents} and  lemma \ref{coseta}, the representations $\varrho_1$ and $\varrho_2$ are not constituents of $\theta_{21}$. Let $p$ be a non-supersingular prime for $E_{63}$ of good reduction for both $E_{63}$ and $J(C_2)$ with residue class degree $f_{L_4}=4$. Thus $L_p(E_{21}/\Q,T)=L_p(E_{63}/\Q,-T)$. By lemma~\ref{calcfac}, the only possible values for a quotient of a root of $L_{p}(E_{21}^3/\Q,T)$ and a root of $L_{p}(J(C_2)/\Q,T)$ of finite order are $-1,i$ and $-i$. Since $\varrho_3(4a_L)$ and $\varrho_4(4a_L)$ have $1$ as an eigenvalue, the representations $\varrho_3$ and $\varrho_4$ are not constituents of $\theta_{21}$, neither. Proceeding analogously, one obtains that $\Tr(\theta_{13})=3\cdot \chi_5'$.

To prove that $\Tr(\theta_{10})=6\cdot \chi_t+ 3\cdot\chi_q $ it is enough to observe that $$\dim_\Q\Hom_\Q(J(C_1),J(C_0))=\dim_\Q\Hom_\Q(E_{21}^3,E_{21}^2\times E_{63})=6$$ and then apply proposition \ref{constituents}.

Proposition \ref{sera} tells us that $\Inf_{\Q(\sqrt{-3})}^{L}\theta_{10}\simeq 6\cdot\varrho_1\oplus 3\cdot\varrho_2$ is isomorphic to a subrepresentation of $\theta(J(C_1),J(C_0);L/\Q)$. But since they both have the same dimension, they are in fact isomorphic. Proposition \ref{trans} tells us that $\theta_{02}$ is isomorphic to a subrepresentation of
    $$\theta_{21}\otimes\theta(J(C_1),J(C_0);L/\Q)\simeq 3\cdot \varrho_5 \otimes (6\cdot\varrho_1\oplus 3\cdot\varrho_2)=18\cdot\varrho_5\oplus 9\cdot\varrho_4\,.$$
Hence, $\Tr(\theta_{02})=n\cdot \chi_4+m\cdot\chi_5 $, for certain integers $0\leq n,\,m\leq 3$ such that $m+n=3$.  Applying again proposition $\ref{trans}$, we have that $\theta(J(C_1),J(C_0);L/\Q)$ is isomorphic to a subrepresentation of
$$
\begin{array}{l@{\,\simeq\,}l}
 \theta_{02}\otimes\theta_{21} & \displaystyle{(3n\cdot \varrho_4\otimes\varrho_5)\oplus (3m\cdot \varrho_5\otimes\varrho_5) }\\
 & \displaystyle{3n (\varrho_2\oplus\varrho_3\oplus\varrho_4\oplus\varrho_5)\oplus 3m(\varrho_1\oplus\varrho_3\oplus\varrho_4\oplus\varrho_5)}\\
 & \displaystyle{3m \cdot \varrho_1\oplus 3n\cdot \varrho_2\oplus 9\cdot\varrho_3\oplus 9\cdot\varrho_4\oplus 9\cdot\varrho_5\,.}
\end{array}
$$
Hence, $m\geq 2$ and $n\geq 1$, which implies $m= 2$ and $n= 1$ as desired. Proceeding analogously, one obtains that $\Tr(\theta_{03})= \chi_4'+2\cdot\chi_5'$.
\end{proof}

Now we can prove that for every prime $p$ of good reduction for both $E_{63}$ and $J(C_2)$, the polynomial $L_p(E_{63}/\Q,T)$ divides $L_p(J(C_2)/\Q,T)$. Observe that
$$
\theta(E_{63}^3,J(C_2);L/\Q)\simeq \varrho_2\otimes\theta(E_{21}^3,J(C_2);L/\Q)\simeq \varrho_4\,.
$$
Since the image of any element in $\Gal(L/\Q)$ by $\varrho_4$ has $1$ as an eigenvalue, the polynomials $L_p(E_{63}^3/\Q,T)$ and $L_p(J(C_2)/\Q,T)$ have a common root for every such $p$. Since $L_p(E_{63}/\Q,T)$ is irreducible, it must divide $L_p(J(C_2)/\Q,T)$.

\begin{prop}\label{megtrans} We have $\Tr(\theta_{32})=\psi_{13}$.
\end{prop}

\begin{proof} First we need to know how the conjugacy classes of $\Gal(L  L'/\Q)$ project onto the conjugacy classes of the quotients $\Gal(L/\Q)$ and $\Gal(L'/\Q)$.  Denote by $\pi_L $ and $\pi_{L'}$ the canonical projections from $\Gal(LL'/\Q)$ to $\Gal(L/\Q)$ and $\Gal(L'/\Q)$, respectively. With the help of Magma (see \cite{Mag}), we obtain:
$$
\begin{array}{|l| l|}\hline
1a_L = \pi_L(1A\cup 2B\cup 3A)  &  1a_{L'}=\pi_{L'}(1A\cup 2A\cup 3B)  \\\hline
2a_L = \pi_L(2A\cup 2C\cup 6A)  &  2a_{L'}=\pi_{L'}(2B\cup 2C\cup 6B)  \\\hline
2b_L = \pi_L(2D\cup 4C)  &  2b_{L'}=\pi_{L'}(2D \cup 4B)  \\\hline
3a_L = \pi_L(3B\cup 3C\cup 3D\cup 6B)  &  3a_{L'}=\pi_{L'}(3A\cup 3C \cup 3D\cup 6A)  \\\hline
4a_L = \pi_L(4A\cup 4B)  &  4a_{L'}=\pi_{L'}(4A \cup 4C)  \\\hline
\end{array}
$$
It is now easy to compute that $\Tr\Inf_L^{L  L'}(\varrho_5)=\psi_{9}$ and that $\Tr\Inf_{L'}^{L  L'}(\varrho_5')=\psi_{7}$. Proposition \ref{sera} says that $\Inf_L^{L  L'}\theta_{21}$ and $\Inf_L^{L  L'}\theta_{13}$ are respectively subrepresentations of $\theta(J(C_2),J(C_1);LL'/\Q)$ and $\theta(J(C_1),J(C_3);LL'/\Q)$. In fact, since they have the same dimension, they coincide. Proposition \ref{trans} states that $\theta_{32}$ is a subrepresentation of
$$\theta(J(C_2),J(C_1);LL'/\Q)\otimes \theta(J(C_1),J(C_3);LL'/\Q)\,.$$
This representation has trace $9\cdot \psi_{9}\cdot \psi_7=9\cdot \psi_{13}$ and now the fact that $\theta_{32}$ has dimension $9$ ensures that $\Tr(\theta_{32})=\psi_{13}$.
\end{proof}

We have seen that $\theta(E_{63}^3,J(C_2);L/\Q)\simeq 3\cdot \varrho_4$. Applying cororllary \ref{collde}, we obtain that $V_\ell(J(C_2))\simeq \varrho_4\otimes V_\ell(E_{63})$, from which the next corollary follows.

\begin{coro}
For every prime $p$ which is unramified in $L$, the local factor $L_p(J(C_2)/\Q,T)$ coincides with the Rankin-Selberg polynomial $L_p(E_{63}/\Q,\varrho_4,T)$.
\end{coro}

\begin{coro} For every prime $p$ unramified in $L$ of good reduction for both $E_{63}$ and $J(C_2)$, we have
$$\#\tilde C_2(\F_{p^r})=(1+p^r)(1-\Tr\varrho_4(\Frob_p^r))+\Tr\varrho_4(\Frob_p^r)\#\tilde E_{63}(\F_{p^r})\,.$$
\end{coro}

\begin{proof} Let $\alpha$ and $\overline\alpha$ be the reciprocals of the roots of $L_p(E_{63}/\Q,T)$ and let $\lambda_i$ denote the roots of $\det(1-\varrho_4(\Frob_p)T)$. Then, we have
$$
\begin{array}{l@{\,=\,}l}
\#\tilde C_2(\F_{p^r}) & \displaystyle{1+p^r-\sum_i(\lambda_i\alpha)^r+(\lambda_i\overline\alpha)^r}\\[6pt]
 & \displaystyle{1+p^r-\Tr\varrho_4(\Frob_p^r)(\alpha^r+\overline\alpha^r)}\\[6pt]
 & \displaystyle{ 1+p^r-\Tr\varrho_4(\Frob_p^r)(1+p^r-\#\tilde E_{63}(\F_{p^r}))\,.}\\[6pt]
\end{array}
$$
\end{proof}

To conclude, we will compute the moments of the Sato-Tate distributions of the traces of the local factor $L_p(J(C_2),T)$. For an abelian variety $A$ over $\Q$ of dimension $g$,  and $p$ a prime of good reduction of $A$, it will be convenient to consider the normalized local factor
$$\overline L_p(A/\Q,T)=L_p(A/\Q,p^{-1/2}T)=\sum_{i=0}^{2g}(-1)^i\overline a_i(p)T^i\,,$$
where $\overline a_i(p)=\overline a_{2g-i}(p),\,|\overline a_i(p)|\leq \binom{2g}{i}\,.$ For $n\geq 0$, let $M_{n}(\overline a_i)=\E(\overline a_i^{n})$ denote the $n$-th moment of $\overline a_i$, that is, the expected value of $\overline a _i^n$, seen as a random variable over the set of primes of $\Q$.
If $A$ is an elliptic curve without complex multiplication, the recent proof of the Sato-Tate Conjecture guarantees that
\begin{equation}\label{momelliptic}
M_{2n}(\overline a_1)=\frac{1}{n+1}\binom{2n}{n}\,,\qquad M_{2n+1}(\overline a_1)=0.
\end{equation}
We will write $c_n= 1/(n+1)\binom{2n}{n}$ for the $n$-th Catalan number (see \cite{KS09} for an account on this and much more). We will write $\overline a=\overline a_1(E_{63})$ and $\overline a _i=\overline a_i(J(C_2))$ for $1\leq i\leq 3$.

\begin{lema} For every prime $p$ of good reduction for both $E_{63}$ and $J(C_2)$:
\begin{enumerate}[i)]
\item $\overline a_1(p) = \overline a(p) \cdot \Tr \varrho_4(\Frob_p)$.
\item $\overline a_2(p) =\Tr\varrho_4(\Frob_p)\cdot (\overline a(p)^2-2+ \Tr\varrho_4(\Frob_p))$.
\item $\overline a_3(p) = \overline a(p)\cdot (\overline a(p) ^2+\Tr\varrho_4(\Frob_p)^2-3$).
\end{enumerate}
\end{lema}

\begin{proof} It follows from the fact that $V_\ell(J(C_2))\simeq V_\ell(E_{63})\otimes \varrho_4$ and that the eigenvalues of $\varrho_4(\Frob_p)$ are $1$, $\zeta_r$ and $\overline \zeta_r$ if $\Frob_p$ has order $r$ in $\Gal(L/\Q)$.
\end{proof}

\begin{prop}Let $G=\Gal(L/\Q)$ and $t_\sigma=\Tr\varrho_4(\sigma)$, for any $\sigma \in G$. For $n\geq 0$, one has that $M_{2n+1}(\overline a_1)=M_{2n+1}(\overline a_3)=0$ and:
\begin{enumerate}[i)]
\item $M_{2n}(\overline a_1)=\frac{1}{|G|}\sum_{\sigma\in G} t_\sigma^{2n}c_n$.
\item $M_{n}(\overline a_2)=\frac{1}{|G|} \sum_{i=0}^n\binom{n}{i} c_i\left(\sum_{\sigma\in G} t_\sigma^{n}(t_\sigma-2)^{n-i}\right)$.
\item $M_{2n}(\overline a_3)=\frac{1}{|G|}\sum_{i=0}^{2n}\binom{2n}{i} c_{i+n}\left(\sum_{\sigma\in G}(t_\sigma^2-3)^{2n-i}\right)$

\end{enumerate}
\end{prop}

\begin{proof}
It follows from the previous lemma, the Cebotarev density theorem, equation (\ref{momelliptic}), and the fact that the restriction of $\overline a$ at the set of primes $p$ for which $\Frob_p=c$, where $c$ is a certain conjugacy class in $G$, has the same distribution as $\overline a$.




\end{proof}

\end{document}